\documentclass[letterpaper, 10 pt, conference]{ieeeconf}  % Comment this line out
\IEEEoverridecommandlockouts                              % This command is only
\usepackage{amsmath}
\usepackage{amssymb,bm} 
\usepackage{theoremref}
\usepackage{apptools}
\usepackage{mathtools}
\usepackage{graphicx}
\usepackage[caption=false, font=footnotesize]{subfig}
\newtheorem{thm}{Theorem}[section]
\newtheorem{lem}{Lemma}[section]
\newtheorem{prop}{Proposition}[section]

\newtheorem{defn}{Definition}[section]
\newtheorem{asmp}{Assumption}[section]

\newtheorem{rem}{Remark}

\title{\LARGE \bf Singularly Perturbed Averaging with Application to Bio-Inspired 3D Source Seeking}

\author{Mahmoud Abdelgalil, Asmaa Eldesoukey, and Haithem Taha% <-this % stops a space
\thanks{This work was supported by NSF Grant CMMI-1846308}% <-this % stops a space
\thanks{M. Abdelgalil and A. Eldesoukey are with the Department of Mechanical and Aerospace Engineering,
        University of California Irvine, Irvine, CA 92617, USA
        {\tt\small maabdelg@uci.edu, aeldesou@uci.edu}}%
\thanks{H. Taha is with the Faculty of Mechanical and Aerospace Engineering, University of California Irvine,
        Irvine, CA 92617, USA
        {\tt\small hetaha@uci.edu}}%
}

\begin{document}
\maketitle
\thispagestyle{empty}
\pagestyle{empty}

%%%%%%%%%%%%%%%%%%%%%%%%%%%%%%%%%%%%%%%%%%%%%%%%%%%%%%%%%%%%%%%%%%%%%%%%%%%%%%%%
\begin{abstract}
We analyze a class of singularly perturbed high-amplitude, high-frequency oscillatory systems that arises in extremum seeking applications. We provide explicit formulas for averaging and establish the convergence of the trajectories of this class of systems to the trajectories of a suitably averaged reduced order system by combining the higher order averaging theorem with singular perturbation techniques. Finally, we propose a novel bio-inspired 3D source seeking algorithm and establish its singular practical stability. 
\end{abstract}
%%%%%%%%%%%%%%%%%%%%%%%%%%%%%%%%%%%%%%%%%%%%%%%%%%%%%%%%%%%%%%%%%%%%%%%%%%%%%%%%
\section{Introduction}
Averaging techniques have been widely used in the construction and the stability analysis of solutions to time varying differential equations \cite{bogoliubov1961, sanders2007averaging, sarychev2001lie, bullo2002averaging, vela2003ageneral,maggia2020higher}. The rigorous application of the method of averaging begins by writing the system on the form:
\begin{align}
    \label{eq:voc_form}\frac{d\bm{\zeta}}{d\tau} & = \textbf{f}_0(\bm{\zeta},\tau) + \varepsilon\,\textbf{f}_1(\bm{\zeta},\tau,\varepsilon)
\end{align}
possibly via coordinate changes and time scaling, where the vector field $\textbf{f}_1$ is periodic in $\tau$ and $\varepsilon$ is a small parameter. When $\textbf{f}_0=0$, the system is said to be on the averaging canonical form and the averaging theorem can be directly applied \cite[Chapter 2]{sanders2007averaging}. If $\textbf{f}_0\neq 0$, the Variation of Constants (VOC) formula may be used to force the system on the averaging canonical form \cite[Section 1.6]{sanders2007averaging},\cite[Section 9.1]{bullo2004geometric}. However, even when the vector field $\textbf{f}_0$ is linear and time invariant, i.e. $\textbf{f}_0(\bm{\zeta},\tau) = \textbf{A}\bm{\zeta}$ for some matrix $\textbf{A}$, the VOC formula is practically useful only when the eigenvalues of the matrix $\textbf{A}$ are purely imaginary. This is due to the fact that the pullback of the vector field $\textbf{f}_1$ under the flow of $\textbf{f}_0$ will contain exponentially growing terms (see the discussion in \cite[Section 1.7]{sanders2007averaging}). 

In this manuscript, we analyze a class of singularly perturbed high-frequency, high-amplitude oscillatory systems described by equation (\ref{eq:orig_sys}) which naturally arises in extremum seeking applications \cite{krstic2000stability, durr2013lie}, and can be put on the form (\ref{eq:voc_form}). Yet, the VOC formula is not useful in analyzing this class of systems for the reasons outlined in the previous paragraph. Moreover, recent results in the literature such as the singularly perturbed Lie Bracket Approximation framework \cite{durr2015singularly, durr2017extremum} do not capture the stability properties as we illustrate below.

To analyze the behavior of this class of systems, we combine the higher order averaging theorem \cite{sanders2007averaging} with singular perturbation techniques \cite{khalil2002nonlinear} in a way that accounts for the interaction between the fast periodic time scale and the singularly perturbed part of the system. Furthermore, we propose a novel 3D source seeking algorithm for rigid bodies with a non-collocated sensor. The proposed algorithm is inspired by the chemotactic strategy of sea urchins sperm cells for seeking the egg in 3D \cite{friedrich2007chemotaxis,abdelgalil2021sperm}, and it utilizes the special structure of the matrix group SO(3). We prove the practical stability of the proposed algorithm using the singularly perturbed averaging results we state here. 
\section{Singularly Perturbed Second Order Averaging}
In this section, we state the main theorem of our work. Consider the interconnection of systems on the form:
\begin{align}
    \label{eq:orig_sys}
    \begin{aligned}
        \dot{\textbf{x}}&= \textbf{f}(\textbf{x},\textbf{y},t,\omega) + \frac{1}{\sqrt{\omega}}\textbf{f}_3(\textbf{x},\textbf{y},t,\omega),  & \textbf{x}(t_0) &= \textbf{x}_0\\ \dot{\textbf{y}}&=\textbf{g}(\textbf{x},\textbf{y},t,\omega)+\frac{1}{\sqrt{\omega}}\textbf{g}_3(\textbf{x},\textbf{y},t,\omega),  & \textbf{y}(t_0) &= \textbf{y}_0
    \end{aligned}
\end{align}
where $\textbf{x},\textbf{x}_0\in\mathbb{R}^n,\,\textbf{y},\textbf{y}_0\in\mathbb{R}^m,\,t,t_0\in\mathbb{R}$, $\omega\in(0,\infty)$, and the maps $\textbf{f},\,\textbf{g}$ are given by:
\begin{gather*}
    \textbf{f}(\textbf{x},\textbf{y},t,\omega) = \sum_{i\in\{1,2\}}\omega^{1-\frac{i}{2}}\, \textbf{f}_i(\textbf{x},\textbf{y},\omega t)\\
    \textbf{g}(\textbf{x},\textbf{y},t,\omega)= \omega \textbf{A}\,(\textbf{y}-\bm{\varphi}_0(\textbf{x})) +\sum_{i\in\{1,2\}}\omega^{1-\frac{i}{2}}\, \textbf{g}_i(\textbf{x},\textbf{y},\omega t)
\end{gather*}
We adopt the following assumptions on the regularity of the right-hand side of equation (\ref{eq:orig_sys}):
\begin{asmp}\thlabel{asmp:A}
Suppose that for $i\in\{1,2\}$:
    \begin{enumerate}
        \item \label{asmp:item_A1}  ${ \textbf{f}_i(\cdot,\cdot,\tau)\in\mathcal{C}^{3-i}(\mathbb{R}^{n+m};\mathbb{R}^n)}$, ${ \textbf{f}_i\in\mathcal{C}^0(\mathbb{R}^{n+m+1};\mathbb{R}^n)}$,
        \item $\textbf{g}_i(\cdot,\cdot,\tau)\in\mathcal{C}^{3-i}(\mathbb{R}^{n+m};\mathbb{R}^n)$, ${ \textbf{g}_i\in\mathcal{C}^0(\mathbb{R}^{n+m+1};\mathbb{R}^n)}$,
        \item \label{asmp:item_A2} $\exists T>0$ s.t. ${ \textbf{f}_i(\cdot,\cdot,\tau+T) = \textbf{f}_i(\cdot,\cdot,,\tau)}$, and $ \textbf{g}_i(\cdot,\cdot,\tau+T)$ $= \textbf{g}_i(\cdot,\cdot,,\tau)$, $\,\forall \tau \in\mathbb{R}$,
        \item $\textbf{f}_3$ and $\textbf{g}_3$ are locally Lipschitz continuous in $\textbf{x},\textbf{y}$ and jointly continuous in all of their arguments,
        \item ${ \int_{0}^{T}\textbf{f}_1(\cdot,\cdot,s)ds} = 0$, $\bm{\varphi}_0(\cdot)\in\mathcal{C}^{3}(\mathbb{R}^n;\mathbb{R}^m)$, and the matrix $ \textbf{A}$ is Hurwitz.
    \end{enumerate}
\end{asmp}
\begin{rem}
We restrict our treatment here to the periodic case for simplicity. The extension to the case when the vector fields are quasi-periodic is straightforward but the computations are more involved.
\end{rem} 
Next, consider the reduced order system:
\begin{gather} \label{eq:reduced_system}
    \dot{\tilde{\textbf{x}}}= \sqrt{\omega}\,\tilde{\textbf{f}}_1(\tilde{\textbf{x}},\omega t )+ \tilde{\textbf{f}}_2(\tilde{\textbf{x}},\omega t),\qquad\qquad
    \tilde{\textbf{y}}= \bm{\varphi}_0(\tilde{\textbf{x}})
\end{gather}
where the time-varying vector fields $\tilde{\textbf{f}}_i$ are defined by:
\begin{gather}
    \label{eq:red_avg_sys_dets_1}\tilde{\textbf{f}}_1(\textbf{x},\tau)= \textbf{f}_1(\textbf{x},\bm{\varphi}_0(\textbf{x}),\tau) \\
    \tilde{\textbf{f}}_2(\textbf{x},\tau)= \textbf{f}_2(\textbf{x},\bm{\varphi}_0(\textbf{x}),\tau) + \textbf{C}(\textbf{x},\bm{\varphi}_0(\textbf{x}),\tau)\bm{\varphi}_1(\textbf{x},\tau)\\
    \textbf{C}(\textbf{x},\textbf{y},\tau)= \partial_\textbf{w} \textbf{f}_1(\textbf{x},\textbf{w},\tau)|_{\textbf{w}=\textbf{y}} \\
    \bm{\varphi}_1(\textbf{x},\tau)= \left(\text{Id}-\text{e}^{T \textbf{A}}\right)^{-1}\textstyle\int_{0}^T \text{e}^{ (T-s)\textbf{A}}\,\textbf{b}_1(\textbf{x},s+\tau)ds \\
    \label{eq:red_avg_sys_dets_5}\textbf{b}_1(\textbf{x},\tau)= \textbf{g}_1(\textbf{x},\bm{\varphi}_0(\textbf{x}),\tau)-\partial_{\textbf{x}}\bm{\varphi}_0(\textbf{x})\, \textbf{f}_1(\textbf{x},\bm{\varphi}_0(\textbf{x}),\tau)
\end{gather}
and $\text{Id}$ is the identity matrix. In a companion paper \cite{abdelgalil2022recursive}, see also \cite{murdock1983some}, we showed that the higher order averaging theorem may be applied to the reduced order system (\ref{eq:reduced_system}) to obtain the reduced order averaged system:
\begin{gather}\label{eq:reduced_avg_system}
    \dot{\bar{\textbf{x}}}= \bar{\textbf{f}}(\bar{\textbf{x}}), \qquad\qquad \bar{\textbf{y}} = \bm{\varphi}_0(\textbf{x})
\end{gather}
where the vector field $\bar{\textbf{f}}(\cdot)$ is given by:
\begin{equation}
\begin{aligned}
    \bar{\textbf{f}}(\textbf{x})= \frac{1}{T}\int_0^{T}\bigg(\tilde{\textbf{f}}_2(\textbf{x},\tau_1)+\frac{1}{2}\Big[ \int_0^{\tau_1}&\tilde{\textbf{f}}_1(\textbf{x}, \tau_2) d\tau_2,\\
    &\tilde{\textbf{f}}_1(\textbf{x},\tau_1)\Big]\bigg)\,d\tau_1
\end{aligned}
\end{equation}
Under \thref{asmp:A}, we have the following theorem concerning the relation between the stability of the system (\ref{eq:orig_sys}) and the reduced order averaged system (\ref{eq:reduced_avg_system}):
\begin{thm}\thlabel{thm:A}
    Let \thref{asmp:A} be satisfied, and suppose that a compact subset $\mathcal{S}\subset\mathbb{R}^n$ is globally uniformly asymptotically stable for the reduced order averaged system (\ref{eq:reduced_avg_system}). Then, $\mathcal{S}$ is singularly semi-globally practically uniformly asymptotically stable for the original system (\ref{eq:orig_sys}).
\end{thm}
\begin{rem}
    We note that our definition of singular semi-global practical uniform asymptotic stability, which can be found in the appendix A, is slightly different from that in \cite{durr2015singularly}. The proof of \thref{thm:A} proceeds by establishing each part of \thref{defn:A} similar to \cite{moreau2000practical,durr2015singularly}, relying on \thref{prop:A} which we can be found along with a proof sketch in the appendix B. We omit the proof of the theorem from this manuscript to be included in the extended version. 
\end{rem}

\iffalse 0
Going back to the motivational example, we see that when $\mu=1/\omega$, the system can be written on the form:
\begin{gather}
    \label{eq:mot_ex_sol_1}\dot{x}_1= \beta x_1 + \sqrt{2\omega}\sin(x_1^2+ x_4-x_3 +\omega t)\\
    \label{eq:mot_ex_sol_2}\dot{x}_2= \beta x_2 + \sqrt{2\omega}\cos(x_1^2+ x_4-x_3 +\omega t)\\
    \label{eq:mot_ex_sol_3}\dot{x}_3=\omega\,(x_4 - x_3), \qquad     \dot{x}_4= \omega\,(x_2^2 -x_4)
\end{gather}
If we apply our results, leaving the details of the computations as an easy exercise for the reader, we obtain the reduced order system:
\begin{align}
\dot{\tilde{x}}_1&= \beta \tilde{x}_1  + \sqrt{2\omega} \sin(\tilde{x}_1^2+\omega t)+ \tilde{x}_2 \sin(2\tilde{x}_1^2+2\omega t)\\
\dot{\tilde{x}}_2&=\alpha \tilde{x}_2  + \sqrt{2\omega} \cos(\tilde{x}_1^2+\omega t)+ \tilde{x}_2 \cos(2\tilde{x}_1^2+2\omega t)
\end{align}
where $\alpha = \beta-1$. By applying the averaging procedure, we obtain the reduced order averaged system:
\begin{align}
\dot{\bar{x}}_1&= (\beta-2) \bar{x}_1 \\
\dot{\bar{x}}_2&= (\beta-1) \bar{x}_2 
\end{align}
It is clear that the reduced order averaged system has a globally uniformly asymptotically stable equilibrium point at the origin when $\beta<1$. Hence, we conclude from \thref{thm:A} that the origin $x_1=x_2=0$ is a singularly semi-globally practically uniformly asymptotically stable subset for the original system (\ref{eq:mot_ex_1})-(\ref{eq:mot_ex_3}). 
\fi
\section{3D Source Seeking}
Source seeking is the problem of locating a target that emits a scalar measurable signal, typically without global positioning information \cite{cochran20093,krstic2008extremum}. Interestingly, microorganisms are routinely faced with the source seeking problem. In particular, sea urchin sperm cells seek the egg by swimming up the gradient of the concentration field  of a chemical secreted by the egg \cite{abdelgalil2021sperm,friedrich2007chemotaxis}. The sperm cells do so by swimming in helical paths that dynamically align with the gradient. In this section, we propose a bio-inspired 3D source seeking algorithm for rigid bodies with a non-collocated signal strength sensor that partially mimics the strategy of sperm cells for seeking the egg.

The kinematics of a rigid body in 3D space are given by:
\begin{align}
    \label{eq:kin}\dot{\textbf{p}} &= \textbf{R}\textbf{v}, & \dot{\textbf{R}} &= \textbf{R}\widehat{\bm{\Omega}}
\end{align}
where $\textbf{p}$ denotes the position of a designated point on the body with respect to a fixed frame of reference, $\textbf{R}$ relates the body frame to the fixed frame, and $\textbf{v}$ and $\bm{\Omega}$ are the linear and angular velocities in body coordinates, respectively. The map $\widehat{\bullet}:\mathbb{R}^3\rightarrow\mathbb{R}^{3\times 3}$ takes a vector $\bm{\Omega}=\left[\Omega_1,\Omega_2,\Omega_3\right]^\intercal \in \mathbb{R}^3$ to the corresponding skew symmetric matrix. 

We assume a vehicle model in which the linear and angular velocity vectors are given in body coordinates by:
\begin{align}
    \textbf{v}&= v\textbf{e}_1, & \bm{\Omega} &= \Omega_\parallel\textbf{e}_1 + \Omega_\perp \textbf{e}_3
\end{align} 
where $\textbf{e}_i$ for $i\in\{1,2,3\}$ are the standard unit vectors. \begin{rem}
This model is a natural extension of the unicycle model to the 3D setting. It is well known that this system is controllable using depth one Lie brackets \cite{leonard1993averaging}.
\end{rem}

Let $c:\mathbb{R}^3\rightarrow\mathbb{R}$ be the signal strength field emitted by the source, and consider the case of a non-collocated signal strength sensor that is mounted at $\textbf{p}_s$, where:
\begin{align}\label{eq:sens_loc}
    \textbf{p}_s&= \textbf{p} + r \textbf{R}\textbf{e}_2
\end{align}
\begin{asmp}\thlabel{asmp:B}
Suppose that the signal strength field $c\in \mathcal{C}^3(\mathbb{R}^3;\mathbb{R})$ is radially unbounded, $\exists!\textbf{p}^*\in \mathbb{R}^3$ such that $\nabla c(\textbf{p})= 0 \iff \textbf{p}=\textbf{p}^*$, and it satisfies $c(\textbf{p}^*)-c(\textbf{p})\leq \kappa \lVert \nabla c(\textbf{p})\lVert^2, \, \forall \textbf{p}\in\mathbb{R}^n$ and $\kappa>0$.\\
\end{asmp}

Now, consider the following control law:
\begin{gather}
\label{eq:ctrl_law_1}v = \sqrt{4\omega}\cos(2\omega t - c(\textbf{p}_s))\\
    \label{eq:ctrl_law_2}\dot{\textbf{y}} = \omega\,\textbf{A}\,\textbf{y} + \omega\,\textbf{B}\, c(\textbf{p}_s) \\
    \Omega_\perp  = \sqrt{\omega}\,\textbf{C}\, \textbf{y}, \qquad
    \label{eq:ctrl_law_3}\Omega_\parallel = \omega
    % \label{eq:ctrl_law_2}v(\textbf{y})&= 2\sqrt{\omega}\cos(2\omega t)+ 2\sqrt{\omega}\sin(2\omega t) \textbf{C}\, \textbf{y}
\end{gather}
where $\textbf{y}\in\mathbb{R}^2$, and $\textbf{A},\textbf{B},\textbf{C}$ are given by:
\begin{align}
    \label{eq:flter_mats}\textbf{A} &= \left[\begin{array}{cc} -1 & 1\\ 0 & -1 \end{array}\right], & \textbf{B}&= \left[\begin{array}{c} 0\\ 1\end{array}\right], & \textbf{C}&= \left[\begin{array}{cc} -1 & 1\end{array}\right]
\end{align}
The static part of this controller, i.e. equation (\ref{eq:ctrl_law_1}) is a 1D extremum seeking control law \cite{scheinker2014extremum}. Note that other choices of this control law are possible \cite{grushkovskaya2018class}. 
The dynamic part of this controller, i.e. the equations (\ref{eq:ctrl_law_2})-(\ref{eq:ctrl_law_3}), is a narrow band-pass filter centered around the frequency $\omega$. The motivation to consider this setup is that in the presence of noise, a narrow band-pass filter centered around the dither frequency $\omega$ optimally extracts the gradient information in the measured output while attenuating noise. In addition, assume that the distance $r$ specifying the offset of the sensor from the center of the frame is such that $r=1/\sqrt{\omega}$. This assumption may seem artificial at first glance, though its implication is clear; we require that as the frequency of oscillation $\omega$ tends to $\infty$, the distance $r$ from the center of the vehicle is small enough so as not to amplify unwanted nonlinearities in the signal strength field. Alternatively, one may consider this assumption as a ``distinguished limit'' \cite{kevorkian2012multiple} for the perturbation calculation in the presence of the two parameters $\omega$ and $r$. Under these assumptions, we have the following proposition:
\begin{prop}
Let \thref{asmp:B} be satisfied, and let $r=1/\sqrt{\omega}$. Then, the compact subset $\mathcal{S}=\{\textbf{p}^*\}\times \text{SO}(3)$ is singularly semi-globally practically uniformly asymptotically stable for the system defined by equations (\ref{eq:kin})-(\ref{eq:flter_mats}).
\end{prop}
\begin{proof}
Let $\textbf{R}_0 =\text{exp}\left(\omega t \,\widehat{\textbf{e}}_1 \right),\,\textbf{Q} = \textbf{R}\textbf{R}_0^\intercal$, and compute:
\begin{align}
    \dot{\textbf{Q}} &= \dot{\textbf{R}}\textbf{R}_0^\intercal +  \textbf{R}\dot{\textbf{R}}_0^\intercal = \Omega_\perp\textbf{R}\widehat{\textbf{e}}_3\textbf{R}_0^\intercal \\
    &=  \Omega_\perp\textbf{Q}\textbf{R}_0 \widehat{\textbf{e}}_3\textbf{R}_0^\intercal = \Omega_\perp\textbf{Q}\widehat{\textbf{R}_0 \textbf{e}_3}
\end{align}
Let $\bm{\Lambda}(\textbf{y},\omega t)= \textbf{C}\, \textbf{y}\,\textbf{R}_0\textbf{e}_3$ and observe that:
\begin{gather}
    \dot{\textbf{p}}= \textbf{R}\textbf{v} = \textbf{R}\textbf{R}_0^\intercal \textbf{R}_0 \textbf{v} = v \,\textbf{Q} \textbf{R}_0 \textbf{e}_1 = v\,\textbf{Q} \textbf{e}_1 \\
    \dot{\textbf{Q}}= \sqrt{\omega}\,\textbf{Q} \widehat{\bm{\Lambda}}(\textbf{y},\omega t)
\end{gather}
To simplify the presentation, we embed $\text{SO}(3)$ into $\mathbb{R}^9$ by partitioning the matrix $\textbf{Q}=[\textbf{q}_1,\, \textbf{q}_2,\, \textbf{q}_3 ]$, and defining the state vector $\textbf{q}=[\textbf{q}_1^\intercal,\, \textbf{q}_2^\intercal,\, \textbf{q}_3^\intercal ]^\intercal$. Restrict the initial conditions for $\textbf{q}$ to lie on the compact submanifold $\mathcal{M} = \left\{\textbf{q}_i\in\mathbb{R}^{3}:\, \textbf{q}_i^\intercal\textbf{q}_j=\delta_{ij},~\textbf{q}_i\times\textbf{q}_j = \epsilon_{ijk}\textbf{q}_k  \right\}$, 
where $\delta_{ij}$ is the Kronecker symbol and $\epsilon_{ijk}$ is the Levi-Civita symbol.
On $\mathbb{R}^3\times\mathcal{M}\times\mathbb{R}^2$, the system is governed by:
\begin{align}
    \label{eq:src_seek_sys_1} \dot{\textbf{p}} &= \sqrt{4\omega}\cos(2\omega t - c(\textbf{p}_s))\,\textbf{q}_1\\
    \label{eq:src_seek_sys_2}\dot{\textbf{q}}_i &= \sqrt{\omega}\sum\limits_{j,k=1}^3\Lambda_j(\textbf{y},\omega t)\epsilon_{ijk}\textbf{q}_k \\
    \label{eq:src_seek_sys_3}\dot{\textbf{y}} &=\omega \left(\textbf{A}\,\textbf{y} + \,\textbf{B}\, c(\textbf{p}_s) \right)
\end{align}
For more details on this embedding, see the companion paper \cite{abdelgalil2022recursive}. The signal strength field can be expanded as a series in $1/\sqrt{\omega}$ using Taylor's theorem:
\begin{gather*}
    c(\textbf{p}_s)=\,c(\textbf{p})+\frac{1}{\sqrt{\omega}} \nabla c(\textbf{p})^\intercal (\cos(\omega t) \textbf{q}_2 + \sin(\omega t) \textbf{q}_3) \\
    +\frac{1}{2\omega }\nabla^2 c(\textbf{p})[(\cos(\omega t) \textbf{q}_2 + \sin(\omega t) \textbf{q}_3)] \\
    + \frac{1}{\omega \sqrt{\omega}}\bm{\rho}(\textbf{p},\textbf{p}_s,\omega t,1/\sqrt{\omega})
\end{gather*}
where the remainder $\bm{\rho}$ is Lipschitz continuous in all of its arguments, and $\nabla^2c(\textbf{p})[\textbf{w}]=\textbf{w}^\intercal \nabla^2c(\textbf{p}) \textbf{w}$. 
Now, observe that the system governed by the equations (\ref{eq:src_seek_sys_1})-(\ref{eq:src_seek_sys_3}) belongs to the class of systems described by (\ref{eq:orig_sys}). Hence, we may employ \thref{thm:A} in analyzing the stability of the system. In order to proceed, the reduced order averaged system needs to be computed. Due to space constraints, we leave the computations as an exercise for the interested reader in the light of equations (\ref{eq:red_avg_sys_dets_1})-(\ref{eq:red_avg_sys_dets_5}), and we provide only the end result of the computation:
\noindent\begin{minipage}{.5\linewidth}
\begin{gather*}
    \dot{\bar{\textbf{p}}} = \vphantom{\frac{1}{4}}\bar{\textbf{q}}_1\bar{\textbf{q}}_1^\intercal \nabla c(\bar{\textbf{p}}) \\
    \dot{\bar{\textbf{q}}}_2 = -\frac{1}{4}\bar{\textbf{q}}_1\bar{\textbf{q}}_3^\intercal \nabla c(\bar{\textbf{p}}),
\end{gather*}
\end{minipage}%
\begin{minipage}{.5\linewidth}
\begin{gather*}
    \dot{\bar{\textbf{q}}}_1= \frac{1}{4}(\bar{\textbf{q}}_2\bar{\textbf{q}}_2^\intercal + \bar{\textbf{q}}_3\bar{\textbf{q}}_3^\intercal)\nabla c(\bar{\textbf{p}}) \\
    \dot{\bar{\textbf{q}}}_3 = -\frac{1}{4}\bar{\textbf{q}}_1\bar{\textbf{q}}_2^\intercal \nabla c(\bar{\textbf{p}})
\end{gather*}
\end{minipage}\vspace*{0.1in} \\
Equivalently, this system can be written as:
\begin{gather}\label{eq:exmp_red_avg_sys_1}
    \dot{\bar{\textbf{p}}} = \bar{\textbf{Q}}\textbf{e}_1\textbf{e}_1^\intercal \bar{\textbf{Q}}^\intercal \nabla c(\bar{\textbf{p}}) \\
    \label{eq:exmp_red_avg_sys_2}\dot{\bar{\textbf{Q}}} = \bar{\textbf{Q}} \widehat{\bar{\bm{\Lambda}}}(\textbf{p},\textbf{Q})
\end{gather}
where the average angular velocity vector $\bar{\bm{\Lambda}}$ is given by:
\begin{align}
    \label{eq:avg_rot_vel} \bar{\bm{\Lambda}}(\textbf{p},\textbf{Q})&= \frac{1}{4}\textbf{Q}^\intercal \nabla c(\textbf{p})\times\textbf{e}_1
\end{align}
We claim that the compact subset $\mathcal{S}$ is globally uniformly asymptotically stable for the reduced order averaged system (\ref{eq:exmp_red_avg_sys_1})-(\ref{eq:exmp_red_avg_sys_2}). To prove this claim, we use the negative of the signal strength field as a Lyapunov function $V_c(\textbf{p})=c(\textbf{p}^*)-c(\textbf{p})$. Observe that the system (\ref{eq:exmp_red_avg_sys_1})-(\ref{eq:exmp_red_avg_sys_2}) is autonomous, and so the function $V_c$ is indeed a Lyapunov function for the compact subset $\mathcal{S}$ due to \thref{asmp:B} \cite{khalil2002nonlinear}. We proceed to compute the derivative of $V_c$:
\begin{gather}
    \dot{V}_c = -\nabla c(\bar{\textbf{p}})^\intercal \bar{\textbf{Q}}\textbf{e}_1\textbf{e}_1^\intercal \bar{\textbf{Q}}^\intercal \nabla c(\bar{\textbf{p}})\leq 0
\end{gather}
Now, consider the subset $\mathcal{N}=\{(\textbf{p},\textbf{Q})\in\mathbb{R}^3\times \text{SO}(3): \dot{V}_c=0\}$, and observe that $\mathcal{S}\subset \mathcal{N}$, and that $\mathcal{S}$ is an invariant subset of the reduced order averaged system (\ref{eq:exmp_red_avg_sys_1})-(\ref{eq:exmp_red_avg_sys_2}). Suppose that a trajectory $(\bar{\textbf{p}}(t),\bar{\textbf{Q}}(t))$ of the system (\ref{eq:exmp_red_avg_sys_1})-(\ref{eq:exmp_red_avg_sys_2}) exists such that $(\bar{\textbf{p}}(t),\bar{\textbf{Q}}(t))\in \mathcal{N}\backslash\mathcal{S},\, \forall t\in I$, where $I$ is the maximal interval of existence and uniqueness of the trajectory. Such a trajectory must satisfy:
\begin{align}
    \nabla c(\bar{\textbf{p}}(t))^\intercal \bar{\textbf{Q}}(t)\textbf{e}_1 = 0,\quad \forall t\in I
\end{align}
The differentiability of the trajectories allows us to compute the derivative of this identity and obtain that:
\begin{align}
    \frac{d}{dt}\big(\nabla c(\bar{\textbf{p}}(t))^\intercal \bar{\textbf{Q}}(t)\textbf{e}_1\big) = 0,\quad \forall t\in I
\end{align}
which simplifies to:
\begin{align}
    \label{eq:second_identity}\nabla c(\bar{\textbf{p}}(t))^\intercal \bar{\textbf{Q}}(t)(\bar{\bm{\Lambda}}(\bar{\textbf{p}}(t),\bar{\textbf{Q}}(t))\times \textbf{e}_1) = 0
\end{align}
Recalling equation (\ref{eq:avg_rot_vel}), we see that:
\begin{gather}
    \bar{\bm{\Lambda}}(\bar{\textbf{p}}(t),\bar{\textbf{Q}}(t))\times \textbf{e}_1 = \frac{1}{4}(\text{Id}-\textbf{e}_1\textbf{e}_1^\intercal)\bar{\textbf{Q}}(t)^\intercal \nabla c(\bar{\textbf{p}}(t))\\
    = \frac{1}{4}\bar{\textbf{Q}}(t)^\intercal \nabla c(\bar{\textbf{p}}(t))
\end{gather}
Hence, the equation (\ref{eq:second_identity}) necessitates that:
\begin{align}
    \lVert \nabla c(\bar{\textbf{p}}(t)) \lVert^2 = 0,\quad \forall t\in I
\end{align}
which is clearly in contradiction with \thref{asmp:B}. Accordingly, it follows from LaSalle's Invariance principle \cite[Corollary 4.2 to Theorem 4.4]{khalil2002nonlinear} that the compact subset $\mathcal{S}$ is globally uniformly asymptotically stable for the system (\ref{eq:exmp_red_avg_sys_1})-(\ref{eq:exmp_red_avg_sys_2}). Hence, we conclude by \thref{thm:A} that the subset $\mathcal{S}$ is singularly semi-globally practically uniformly asymptotically stable for the original system defined by (\ref{eq:kin})-(\ref{eq:flter_mats}). 
\end{proof}
\begin{figure*}[t]
    \centering \includegraphics[width=1\textwidth]{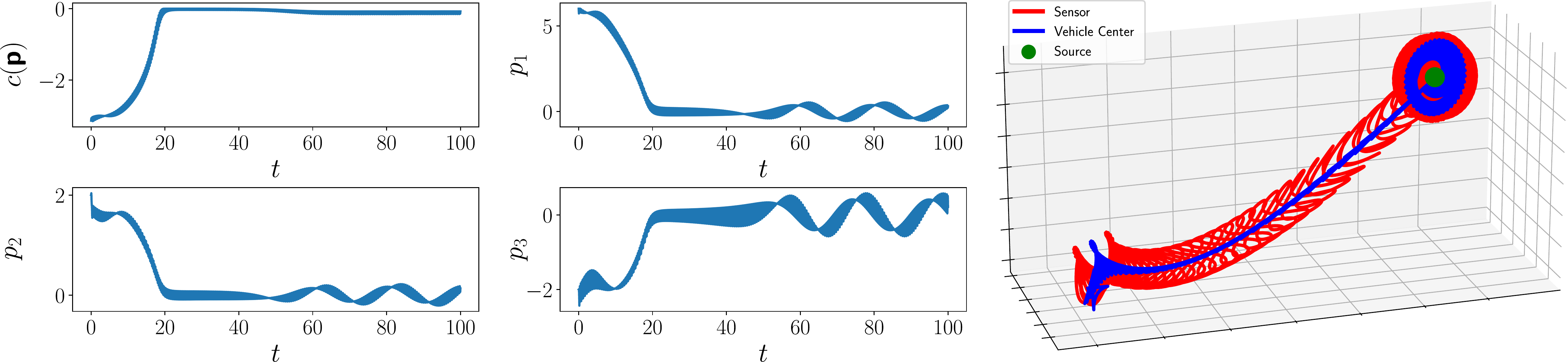}
    \caption{Numerical results: the history of the signal strength at the vehicle center and the position coordinates (left), and the 3D spatial trajectory (right)}
    \label{fig:num_res}
\end{figure*}
\begin{rem}
If we attempt to apply the framework of singularly pertubed Lie Bracket Approximation introduced in \cite{durr2015singularly} to the system (\ref{eq:src_seek_sys_1})-(\ref{eq:src_seek_sys_3}), then the quasi-steady state of the system will be $\textbf{y} = [c(\textbf{p}_s),c(\textbf{p}_s)]^\intercal$. Hence, according to \cite{durr2015singularly}, the reduced order system is:
\begin{gather}
    \dot{\textbf{p}} = \sqrt{4\omega}\cos(2\omega t - c(\textbf{p}_s))\,\textbf{q}_1,\qquad 
    \dot{\textbf{q}}_i =0
\end{gather}
which yields the Lie Bracket system:
\begin{gather}
    \dot{\bar{\textbf{p}}} = \bar{\textbf{Q}}\textbf{e}_1\textbf{e}_1^\intercal \bar{\textbf{Q}}^\intercal \nabla c(\bar{\textbf{p}}), \qquad 
    \dot{\bar{\textbf{Q}}} =0
\end{gather}
It is clear that the compact subset $\mathcal{S}$ is not asymptotically stable for the Lie Bracket system, and so the framework in \cite{durr2015singularly} does not capture the stability of this system. \end{rem}
\section{Numerical Simulations}
We demonstrate our results by providing a numerical example. Consider the signal strength field given by $c(\textbf{p})= -\log\left(1+\textbf{p}^\intercal\textbf{p}/2\right)$, which represents a stationary source located at the origin. We take the initial conditions as $\textbf{p}(0)=[6,~2,~-2]^\intercal$ and $\textbf{R}(0)=\textbf{I}_{3\times3}$, and the frequency as $\omega=4\pi$. The numerical simulations are shown in Fig.(\ref{fig:num_res}). Observe that the behavior near the source is nontrivial, i.e. there is a limit cycle. However, in the limit $\omega\rightarrow \infty$ and $r=O(1/\sqrt{\omega})$, this complex behavior does not appear in the reduced order averaged system.

\section{Conclusion}
In this manuscript, we analyzed a class of singularly perturbed high-amplitude, high-frequency oscillatory systems that naturally arises in extremum seeking applications and stabilization by oscillatory controls. We combined singular perturbation with the higher order averaging theorem in order to capture the stability properties of this class of systems. As an application, we proposed a novel 3D source seeking algorithm for rigid bodies with a non-collocated sensor inspired by the chemotaxis of sea urchin sperm cells.
\section*{Acknowledgement} The authors like to acknowledge the support of the NSF Grant CMMI-1846308. \\
\appendices
\renewcommand{\thelem}{\thesection.\arabic{lem}} 
\renewcommand{\thecor}{\thesection.\arabic{cor}}
\renewcommand{\thedefn}{\thesection.\arabic{defn}} 
\renewcommand{\theprop}{\thesection.\arabic{prop}} 
\section{Definitions}
\begin{defn}\thlabel{defn:A}
 The set $\mathcal{S}$ is said to be \textit{singularly semi-globally practically uniformly asymptotically stable} for system (\ref{eq:orig_sys}) if the following is satisfied:
\begin{enumerate}
    \item $\forall \epsilon_x,\epsilon_z\in(0,\infty)$ there exists $\delta_x,\delta_z\in(0,\infty)$ and $\omega^*\in(0,\infty)$ such that $\forall \omega\in(\omega^*,\infty)$, $\forall t_0\in\mathbb{R}$, and $\forall t\in[0,\infty)$ we have: 
\begin{align*}
\begin{drcases}
     \hfill \textbf{x}_0\in \mathcal{U}_{\delta_x}^{\mathcal{S}}  \hfill \\
     \textbf{y}_0-\bm{\varphi}_0(\textbf{x}_0)\in \mathcal{U}_{\delta_z}^{0}
\end{drcases} &\implies \begin{cases} 
     \hfill\textbf{x}(t)\in \mathcal{U}_{\epsilon_x}^{\mathcal{S}}\hfill \\
     \textbf{y}(t)-\bm{\varphi}_0(\textbf{x}(t))\in \mathcal{U}_{\epsilon_z}^{0} 
\end{cases}
\end{align*}
\item $\forall \epsilon_x,\epsilon_z\in(0,\infty)$ and all $\delta_x,\delta_z\in(0,\infty)$, there exists a time $T_f\in(0,\infty)$ and $\omega^*\in(0,\infty)$ such that $\forall \omega\in(\omega^*,\infty)$, $\forall t_0\in\mathbb{R}$, $\forall t_1\in[T_f,\infty)$, and $\forall t_2\in[T_f/\omega,\infty)$ we have: 
\begin{align*}
\begin{drcases}
     \hfill \textbf{x}_0\in \mathcal{U}_{\delta_x}^{\mathcal{S}}  \hfill\\
     \textbf{y}_0-\bm{\varphi}_0(\textbf{x}_0)\in \mathcal{U}_{\delta_z}^{0}
\end{drcases} &\implies \begin{cases} 
     \hfill \textbf{x}(t_1)\in \mathcal{U}_{\epsilon_x}^{\mathcal{S}} \hfill \\
     \textbf{y}(t_2)-\bm{\varphi}_0(\textbf{x}(t_2))\in \mathcal{U}_{\epsilon_z}^{0}
\end{cases}
\end{align*}
\item $\forall \delta_x,\delta_z\in(0,\infty)$ there exists $\epsilon_x,\epsilon_z\in(0,\infty)$ and $\omega^*\in(0,\infty)$ such that $\forall \omega\in(\omega^*,\infty)$, $\forall t_0\in\mathbb{R}$, and $\forall t\in[0,\infty)$ we have: 
\begin{align*}
\begin{drcases}
     \hfill \textbf{x}_0\in \mathcal{U}_{\delta_x}^{\mathcal{S}}  \hfill\\
     \textbf{y}_0-\bm{\varphi}_0(\textbf{x}_0)\in \mathcal{U}_{\delta_z}^{0}
\end{drcases} &\implies \begin{cases} 
     \hfill \textbf{x}(t)\in \mathcal{U}_{\epsilon_x}^{\mathcal{S}}\hfill \\
     \textbf{y}(t)-\bm{\varphi}_0(\textbf{x}(t))\in \mathcal{U}_{\epsilon_z}^{0}
\end{cases}\\
\end{align*}
\end{enumerate}
\end{defn}
Observe that our definition of singular semi-global practical uniform asymptotic stability is different from the definitions introduced in \cite{durr2015singularly} due to the absence of a second parameter. Intuitively, when two or more parameters are involved, the so called `distinguished limit' is a standard technique in perturbation theory that simplifies the interaction between the limiting behavior of the two parameters on the trajectories of differential equations \cite[Chapter 2]{kevorkian2012multiple}. Our definitions are motivated by this concept of distinguished limits. 
\section{Trajectory Approximation}
For the purpose of brevity, we state here some notations that may enhance the readability of the proof. Whenever a Lipschitz property of a map $\textbf{f}$ over a subset $\mathcal{K}$ is employed, the corresponding Lipschitz constant is labelled as $L_{\textbf{f},\mathcal{K}}$. Similarly, when a uniform bound is employed, it is labelled as $B_{\textbf{f},\mathcal{K}}$. Sometimes we use $M_{\textbf{f},\mathcal{K}}$ as a generic constant when a mix of the two properties is used. Finally, we may omit mentioning the map in the constant label when it is too long or when it is clear from the context. 

Under \thref{asmp:A}, we have a trajectory approximation result between the original system (\ref{eq:orig_sys}) and the reduced order averaged system (\ref{eq:reduced_avg_system}): 
\begin{prop}\thlabel{prop:A}
Let \thref{asmp:A} be satisfied, and suppose that a compact subset $\mathcal{S}\subset\mathbb{R}^n$ is globally uniformly asymptotically stable for the averaged reduced order system (\ref{eq:reduced_avg_system}). Then, there exist constants $\lambda >0 $ and $\gamma > 0$ such that for every bounded subset $ \mathcal{B}_{\textbf{x}}\times \mathcal{B}_{\textbf{z}}\subset\mathbb{R}^n\times\mathbb{R}^m$, $\forall t_f\in(0,\infty)$, and $\forall D\in(0,\infty)$, there exists $\omega^*\in(0,\infty)$ such that $\forall \omega\in(\omega^*,\infty)$, $\forall t_0\in\mathbb{R}$, $ \forall (\textbf{x}_0,\textbf{y}_0-\bm{\varphi}_0(\textbf{x}_0))\in \mathcal{B}_{\textbf{x}}\times \mathcal{B}_{\textbf{z}}$, and $\forall t\in[t_0,t_0+t_f]$, unique trajectories of the system (\ref{eq:orig_sys}) exist and satisfy:
\begin{gather}
    \lVert \textbf{x}(t)-\bar{\textbf{x}}(t) \lVert < D \\
     \lVert \textbf{y}(t)-\bm{\varphi}_0(\textbf{x}(t))\lVert < \gamma \lVert \textbf{y}_0-\bm{\varphi}_0(\textbf{x}_0)\lVert \,\text{e}^{-\omega(t-t_0)\lambda } + D
\end{gather}
\end{prop}
\begin{proof}
We apply the time scaling $\tau=\omega (t-t_0)$, and we let $\varepsilon=1/\sqrt{\omega}$. In contrast to the standard singular perturbation analysis which starts with a coordinate shift for the singularly perturbed part of the system from $\textbf{y}$ to $\textbf{y}-\bm{\varphi}_0(\textbf{x})$ (e.g. \cite[Chapter 11]{khalil2002nonlinear}, \cite[Section I]{durr2015singularly}), we augment the standard coordinate shift with a \textit{near-identity} part:
\begin{equation}\label{eq:nearid_shift}
\textbf{z}=\textbf{y}-\bm{\varphi}_0(\textbf{x})-\varepsilon\, \bm{\varphi}_1(\textbf{x},\tau)-\varepsilon^2\bm{\varphi}_2(\textbf{x},\tau),
\end{equation}
where the maps $\bm{\varphi}_i(\textbf{x},\tau)$ for $i\in\{1,2\}$ are yet to be determined. This is coordinate shift is inspired by the standard near identity transform common in the higher order averaging literature \cite[Section 2.8]{sanders2007averaging}. Observe that under this coordinate and time scale change, we have:
\begin{equation}
\begin{aligned}\label{eq:shifted_orig_sys}
    \frac{d\textbf{x}}{d\tau}&=\sum_{i=1}^2 \varepsilon^i \textbf{v}_i(\textbf{x},\textbf{z},\tau) +\varepsilon^3 \textbf{v}_3(\textbf{x},\textbf{z},\tau,\varepsilon)\\
    \frac{d\textbf{z}}{d\tau}&= \textbf{A}\,\textbf{z} + \sum_{i=1}^2 \varepsilon^i\textbf{h}_i(\textbf{x},\textbf{z},\tau) + \varepsilon^3 \textbf{h}_3(\textbf{x},\textbf{z},\tau,\varepsilon)
\end{aligned}
\end{equation}
where the vector fields $\textbf{v}_i$ and $\textbf{h}_i$ for $i\in\{1,2\}$ are given by:
\begin{align}
    &\begin{aligned}
    \textbf{v}_1(\textbf{x},\textbf{z},\tau)&= \textbf{f}_1(\textbf{x},\textbf{z}+\bm{\varphi}_0(\textbf{x}),\tau)
    \end{aligned}\\
    &\begin{aligned}
    \textbf{v}_2(\textbf{x},\textbf{z},\tau)&= \textbf{f}_2(\textbf{x},\textbf{z}+\bm{\varphi}_0(\textbf{x}),\tau)\\
     &+\textbf{C}(\textbf{x},\textbf{z}+\bm{\varphi}_0(\textbf{x}),\tau) \bm{\varphi}_1(\textbf{x},\tau)
    \end{aligned}\\
    &\begin{aligned}
    \textbf{h}_1(\textbf{x},\textbf{z},\tau)&= \textbf{A}\,\bm{\varphi}_1(\textbf{x},\tau)-\partial_\tau\bm{\varphi}_1(\textbf{x},\tau) \\
    &- \partial_{\textbf{x}} \bm{\varphi}_0(\textbf{x}) \textbf{f}_1(\textbf{x},\textbf{z}+\bm{\varphi}_0(\textbf{x}),\tau)\\
    &+\textbf{g}_1(\textbf{x},\textbf{z}+\bm{\varphi}_0(\textbf{x}),\tau)
    \end{aligned}\\
    &\begin{aligned}
    \textbf{h}_2(\textbf{x},\textbf{z},\tau)&=\textbf{A}\, \bm{\varphi}_2(\textbf{x},\tau)-\partial_\tau\bm{\varphi}_2(\textbf{x},\tau)\\
    &-\partial_{\textbf{x}} \bm{\varphi}_1(\textbf{x},\tau) \textbf{f}_1(\textbf{x},\textbf{z}+\bm{\varphi}_0(\textbf{x}),\tau) \\
    &- \partial_{\textbf{x}} \bm{\varphi}_0(\textbf{x}) \textbf{f}_2(\textbf{x},\textbf{z}+\bm{\varphi}_0(\textbf{x}),\tau)\\
    &+\partial_{\textbf{w}} \textbf{g}_1(\textbf{x},\textbf{w},\tau)|_{\textbf{w}=\bm{\varphi}_0(\textbf{x})}\bm{\varphi}_1(\textbf{x},\tau) \\
    &+\textbf{g}_2(\textbf{x},\textbf{z}+\bm{\varphi}_0(\textbf{x}),\tau)
    \end{aligned}
\end{align}
Now, we let $\bm{\varphi}_1(\textbf{x},\tau)$ and $\bm{\varphi}_2(\textbf{x},\tau)$ be the solutions of the linear non-homogeneous two point boundary value problems:
\begin{align}\label{eq:BVPi_1}
    \partial_\tau\bm{\varphi}_i(\textbf{x},\tau)&= \textbf{A}\, \bm{\varphi}_i(\textbf{x},\tau) + \textbf{b}_i(\textbf{x},\tau)\\ \label{eq:BVPi_2}
    \bm{\varphi}_i(\textbf{x},\tau) &= \bm{\varphi}_i(\textbf{x},\tau+T)
\end{align}
for $i\in\{1,2\}$, where:
\begin{align}
     &\textbf{b}_1(\textbf{x},\tau)= \textbf{g}_1(\textbf{x},\bm{\varphi}_0(\textbf{x}),\tau)-\partial_{\textbf{x}} \bm{\varphi}_0(\textbf{x}) \textbf{f}_1(\textbf{x},\bm{\varphi}_0(\textbf{x}),\tau)\\
     &\begin{aligned}
     \textbf{b}_2(\textbf{x},\tau)&= \textbf{g}_2(\textbf{x},\bm{\varphi}_0(\textbf{x}),\tau)-\partial_{\textbf{x}} \bm{\varphi}_0(\textbf{x}) \textbf{f}_2(\textbf{x},\bm{\varphi}_0(\textbf{x}),\tau)\\
     &-\partial_{\textbf{x}} \bm{\varphi}_1(\textbf{x},\tau) \textbf{f}_1(\textbf{x},\bm{\varphi}_0(\textbf{x}),\tau) \\
     &+\partial_{\textbf{w}} \textbf{g}_1(\textbf{x},\textbf{w},\tau)|_{\textbf{w}=\bm{\varphi}_0(\textbf{x})}\bm{\varphi}_1(\textbf{x},\tau)
     \end{aligned}
\end{align}
The following lemma is a simple consequence of \thref{asmp:A} and standard linear systems theory:
\begin{lem}\thlabel{lem:B}
Let \thref{asmp:A} be satisfied. Then, the non-homogeneous BVPs (\ref{eq:BVPi_1})-(\ref{eq:BVPi_2}) have unique solutions $\bm{\varphi}_i\in\mathcal{C}^{3-i}(\mathbb{R}^{n};\mathbb{R}^m)$ defined by:
\begin{align}
    \bm{\varphi}_i(\textbf{x},\tau)&= \left(\text{Id}-\text{e}^{T \textbf{A}}\right)^{-1}\int_{0}^T \text{e}^{ (T-s)\textbf{A}}\,\textbf{b}_i(\textbf{x},s+\tau)ds
\end{align}
where $\text{Id}$ is the identity map on $\mathbb{R}^m$. 
\end{lem}
\begin{proof}
    The result can be verified by direct substitution, and the regularity of the solutions follows from \thref{asmp:A}.
\end{proof}

With this choice of the maps $\bm{\varphi}_i(\textbf{x},\tau)$ for $i\in\{1,2\}$, observe that $\textbf{v}_1(\textbf{x},0,\tau)=\tilde{\textbf{f}}_1(\textbf{x},\tau),\,\textbf{v}_2(\textbf{x},0,\tau)=\tilde{\textbf{f}}_2(\textbf{x},\tau)$, and that $\textbf{h}_1(\textbf{x},0,\tau)=\textbf{h}_2(\textbf{x},0,\tau)=0$, $\forall \textbf{x}\in\mathbb{R}^n,\,\forall \tau\in\mathbb{R}$. That is, the origin $\textbf{z}=0$ is an equilibrium point for the boundary layer model:
\begin{align}
    \label{eq:bl_model_z}\frac{d\textbf{z}}{d\tau} &= \textbf{A}\,\textbf{z} + \sum_{i=1}^2 \varepsilon^i\textbf{h}_i(\textbf{x},\textbf{z},\tau), & \textbf{z}(0)&= \textbf{z}_0
\end{align}
Moreover, it can be shown that the vector fields $\textbf{h}_i$ for $i\in\{1,2\}$ are Lipschitz continuous and bounded on every compact subset $\mathcal{K}\subset\mathbb{R}^n\times\mathbb{R}^m$, uniformly in $\tau$, for some Lipschitz constants $L_{\textbf{h}_i,\mathcal{K}}>0$ and bounds $B_{\textbf{h}_i,\mathcal{K}}>0$, and that the remainder terms $\textbf{h}_3$ and $\textbf{v}_3$ are continuous and bounded on any compact subset $\mathcal{K}\subset\mathbb{R}^n\times\mathbb{R}^m$ uniformly in $\tau\in\mathbb{R}$ and $\varepsilon\in[0,\varepsilon_0]$ for some $\varepsilon_0>0$. Next, we have the following lemma:
\begin{lem}\thlabel{lem:A}
Let \thref{asmp:A} be satisfied, and suppose that a compact subset $\mathcal{S}\subset\mathbb{R}^n$ is globally uniformly asymptotically stable for the averaged reduced order system (\ref{eq:reduced_avg_system}). Then, there exist constants $\lambda >0 $ and $\gamma > 0$ such that for every bounded subset $ \mathcal{B}_{\textbf{x}}\times \mathcal{B}_{\textbf{z}}\subset\mathbb{R}^n\times\mathbb{R}^m$, $\forall t_f\in(0,\infty)$, and $\forall D\in(0,\infty)$, there exists $\varepsilon^*\in(0,\varepsilon_0)$ such that $\forall \varepsilon\in(0,\varepsilon^*)$,  $ \forall (\textbf{x}_0,\textbf{z}_0)\in \mathcal{B}_{\textbf{x}}\times \mathcal{B}_{\textbf{z}}$, and $\forall \tau\in[0,t_f/\varepsilon^{2}]$, unique trajectories of the system (\ref{eq:shifted_orig_sys}) exist and satisfy:
\begin{gather}
    \lVert \textbf{x}(\tau)-\tilde{\textbf{x}}(\tau) \lVert < D \\
     \lVert \textbf{z}(\tau)\lVert < \gamma \lVert \textbf{z}_0\lVert \,\text{e}^{-\lambda\,\tau} + D
\end{gather}
\end{lem}
\begin{proof}
The full proof of this Lemma is rather long. So we include it here for review purposes, but it will be replaced by a sketch in the final manuscript to conform with the page limit. The full proof will appear in a journal version of the current manuscript. The proof combines ideas from \cite[Lemma 1]{durr2015singularly} and \cite[Lemma 2.8.2]{sanders2007averaging}. 

Fix an arbitrary bounded subset $\mathcal{B}_{\textbf{x}}\times \mathcal{B}_{\textbf{z}}\subset\mathbb{R}^n\times\mathbb{R}^m$, an arbitrary $D\in(0,\infty)$, and an arbitrary $t_f\in (0,\infty)$. Due to \thref{asmp:A}, we know that  $\forall (\textbf{x}_0,\textbf{y}_0)\in \mathcal{B}_{\textbf{x}}\times \mathcal{B}_{\textbf{z}}$, and $\forall \varepsilon\in(0,\varepsilon_0]$, unique trajectories of the system (\ref{eq:shifted_orig_sys}) exist.  Let $[0,\tau_e)$ with $\tau_e\in(0,\infty)$ be the maximal interval of existence and uniqueness of a given solution $(\textbf{x}(\tau),\textbf{z}(\tau))$, where the dependence of the solution on the initial condition is suppressed for brevity. By assumption, we know that the compact subset $\mathcal{S}$ is globally uniformly asymptotically stable for the reduced order averaged system (\ref{eq:reduced_avg_system}). Hence, according to \cite{abdelgalil2022recursive,durr2013lie}, we know that the compact subset $\mathcal{S}$ is semi-globally practically uniformly asymptotically stable for the reduced order system (\ref{eq:reduced_system}). That is, we know that $\exists \varepsilon_1\in (0,\varepsilon_0)$ and a compact subset $\mathcal{N}_{\textbf{x}}\subset \mathbb{R}^n$ such that $\forall \textbf{x}_0\in \mathcal{B}_{\textbf{x}}$, and $\forall \varepsilon\in(0,\varepsilon_1)$, solutions $\tilde{\textbf{x}}(\tau)$ to system (\ref{eq:reduced_system}) exist on the interval $[0,\infty)$, and $\tilde{\textbf{x}}(\tau)\in \mathcal{N}_{\textbf{x}},\,\forall \tau\in [0,\infty)$. Define an open tubular neighborhood $\mathcal{O}_\textbf{x}(\tau)$ around $\tilde{\textbf{x}}(\tau)$ by $\mathcal{O}_\textbf{x}(\tau)=\{\textbf{x}\in\mathbb{R}^n\,:\,\lVert \textbf{x} - \tilde{\textbf{x}}(\tau)\lVert < D\}$, and observe that the $\textbf{x}$-component of the solution to (\ref{eq:shifted_orig_sys}) is initially inside $\mathcal{O}_\textbf{x}(0)$, i.e. $\textbf{x}(0)=\textbf{x}_0\in \mathcal{O}_\textbf{x}(0)$. Moreover, define the compact subset $\mathcal{M}_\textbf{x}=\overline{\{x\in\mathbb{R}^n:\,\inf_{\textbf{x}'\in \mathcal{N}_\textbf{x}}\lVert \textbf{x}-\textbf{x}'\lVert <D\}}$, where the overline indicates the closure of a set. The continuity of the solutions $(\textbf{x}(\tau),\textbf{z}(\tau))$ implies that one of the following cases holds: C1) $\exists \tau_D\in (0,\tau_e)$ s.t. $\textbf{x}(\tau)\in \mathcal{O}_\textbf{x}(\tau),\,\forall \tau\in[0,\tau_D)$, and $\lVert \textbf{x}(\tau_D)-\tilde{\textbf{x}}(\tau_D)\lVert = D$, or C2) $\textbf{x}(\tau)\in \mathcal{O}_\textbf{x}(\tau),\,\forall \tau\in [0,\tau_e)$. Suppose that case C1) holds and observe that $\tilde{\textbf{x}}(\tau)\in \mathcal{M}_\textbf{x},\,\forall \tau\in[0,\tau_D]$. From \thref{asmp:A}, we know that the matrix $\textbf{A}$ is Hurwitz, which implies \cite{khalil2002nonlinear} that there exists a function $ V\in\mathcal{C}^1(\mathbb{R}^m;\mathbb{R})$, and positive constants $\alpha_i,\,i\in\{1,2,3,4\}$ such that:
        \begin{gather}
            \label{eq:Lyapunov_1}\alpha_1\lVert \textbf{z} \lVert^2 \leq V(\textbf{z}) \leq \alpha_2 \lVert \textbf{z} \lVert ^2 \\
            \label{eq:Lyapunov_2}\partial_{\textbf{z}}V(\textbf{z}) \textbf{A}\,\textbf{z} \leq -\alpha_3 \lVert \textbf{z} \lVert ^2, \quad \left\lVert \partial_{\textbf{z}}V(\textbf{z})  \right\lVert \leq \alpha_4 \lVert \textbf{z} \lVert
        \end{gather}
    
    Let $c\in (0,\infty)$ be such that the compact subset $\mathcal{N}_\textbf{z}=\{\textbf{z}\in\mathbb{R}^m :\, \lVert \textbf{z} \lVert \leq \sqrt{c/\alpha_2} \}$ contains the bounded set $\mathcal{B}_{\textbf{z}}$, and define the compact subset $\mathcal{M}_\textbf{z}=\{z\in\mathbb{R}^m :\, \lVert z \lVert \leq \sqrt{c/\alpha_1} \}$. Furthermore, define the compact subset $\mathcal{K}=\mathcal{M}_\textbf{x}\times \mathcal{M}_\textbf{z}$, and compute the derivative of $V$ along the trajectories of (\ref{eq:shifted_orig_sys}):
    \begin{align}
        \frac{dV}{d\tau} &= \partial_{\textbf{z}}V(\textbf{z}) \textbf{A}\,\textbf{z} +\sum_{i=1}^2 \varepsilon^i\partial_{\textbf{z}}V(\textbf{z})  \textbf{h}_i(\textbf{x},\textbf{z},\tau)+O(\varepsilon^3)\\
        &\leq -\left(\alpha_3- \alpha_4\sum_{i=1}^2 \varepsilon^i L_{\textbf{h}_i,\mathcal{K}}\right) \lVert \textbf{z} \lVert ^2 + B_{\textbf{h}_3,\mathcal{K}}\,\varepsilon^3
    \end{align}
    where we used the fact that the $O(\varepsilon^3)$ terms are uniformly bounded on the compact subset $\mathcal{K}$, and the vector fields $\textbf{h}_1$ and $\textbf{h}_2$ have an equilibrium point at $\textbf{z}=0$ and are Lipschitz continuous. Let $\varepsilon_2=\min\{\varepsilon_1,1,\alpha_3/(2\alpha_4(L_{h_1,\mathcal{K}}+L_{h^2,\mathcal{K}}))\}$, and observe that $\forall \varepsilon\in(0,\varepsilon_2)$, $\forall (\textbf{x},\textbf{z})\in \mathcal{K}$, $\forall \tau\in \mathbb{R}$, we have:
    \begin{align}
        \frac{d V(\textbf{z})}{d\tau} \leq -\frac{\alpha_3}{2}\lVert \textbf{z} \lVert ^2 + B_{\textbf{h}_3,\mathcal{K}}\, \varepsilon^3
    \end{align}
    Let $\varepsilon_3=\min\{\varepsilon_2,\sqrt{\alpha_3 c /(4\alpha_2 B_{\textbf{h}_3,\mathcal{K}})}\}$, then observe that $\forall\varepsilon\in(0,\varepsilon_3),\,\forall \tau\in\mathbb{R}$, and $\forall (\textbf{x},\textbf{z})\in \mathcal{M}_\textbf{x}\times \mathcal{M}_\textbf{z}\backslash \mathcal{N}_\textbf{z}$, we have that $\dot{V}\leq 0$, and so the solution $(\textbf{x}(\tau),\textbf{z}(\tau))$ stays inside $\mathcal{M}_\textbf{x}\times \mathcal{M}_\textbf{z}$, $\forall \tau\in[0,\tau_D]$. In addition, similar to \cite[Theorem 4.18]{khalil2002nonlinear}, we have that there exists constants $\gamma,\lambda,\alpha >0$ such that the estimate $\lVert \textbf{z}(\tau)\lVert < \gamma\,\lVert \textbf{z}_0\lVert \text{e}^{-\lambda \tau} + \alpha\,\varepsilon^{\frac{3}{2}}$, holds on the time interval $\tau\in[0,\tau_D]$, $\forall (\textbf{x}_0,\textbf{z}_0)\in\mathcal{B}_\textbf{x}\times \mathcal{B}_\textbf{z}$, $\forall \varepsilon\in(0,\varepsilon_3)$. We emphasize that the constants $\gamma, \lambda, \alpha$ depend on the constants $\alpha_j,L_{\textbf{h}_i,\mathcal{K}},B_{\textbf{h}_i,\mathcal{K}}$, but do not depend on the choice of $\varepsilon\in(0,\varepsilon_3)$. Similar arguments can be used to establish that in C2), we have that $(\textbf{x}(\tau),\textbf{z}(\tau))\in \mathcal{M}_\textbf{x}\times \mathcal{M}_\textbf{z},\,\forall \tau\in [0,\tau_e)$, which implies that $[0,\infty)\subset [0,\tau_e)$. Now observe that we may choose $\varepsilon<(D/\alpha)^{\frac23}$ in to ensure that the result of the lemma holds in case C2) (see also the proof of Lemma 1 in \cite{durr2015singularly}). 
    
    Next, we define an $\varepsilon$-dependent time $\tau_\varepsilon$ by requiring that the following inequality is satisfied:
    \begin{align}
        \label{eq:t_epsilon} \gamma\,\lVert \textbf{z}_0\lVert \text{e}^{-\lambda\,\tau} &< \alpha\, \varepsilon^{\frac32}, & \forall \textbf{z}_0&\in \mathcal{M}_\textbf{z},\,\forall \tau>\tau_\varepsilon
    \end{align}
    and observe that this is always possible for $\varepsilon >0$. In fact, it can be shown that $\tau_\varepsilon=$ $\max\{(3/(2\lambda)) $ $ \log((\gamma \sqrt{c/\alpha_2})/(\alpha\,\varepsilon)),$ $0\}$ satisfies the inequality (\ref{eq:t_epsilon}).  Now, we show that $\exists \varepsilon_4\in(0,\varepsilon_3)$ such that $\tau_\varepsilon<\tau_D,\, \forall \varepsilon\in (0,\varepsilon_4)$. To obtain a contradiction, suppose that there exists a bounded subset $\mathcal{B}_{\textbf{x}}\times \mathcal{B}_{\textbf{z}}\subset\mathbb{R}^n\times \mathbb{R}^m$, and a $D\in(0,\infty)$, such that $\forall \varepsilon_4\in(0,\varepsilon_3)$, $\exists \varepsilon\in(0,\varepsilon_4)$ such that $\tau_\varepsilon \geq \tau_D$. We estimate the difference $\lVert \textbf{x}(\tau_D)-\tilde{\textbf{x}}(\tau_D)\lVert = \lVert \int_{0}^{\tau_D}\big(\sum_{i=1}^2\varepsilon^i (\textbf{v}_i(\textbf{x}(\tau),\textbf{z}(\tau),\tau))-\tilde{\textbf{f}}_i(\tilde{\textbf{x}}(\tau),\tau) + O(\varepsilon^3)\big) d\tau\lVert \leq \varepsilon \int_{0}^{\tau_D} B_{\textbf{v}+\textbf{f},\mathcal{K}} d\tau \leq B_{\textbf{v}+\textbf{f},\mathcal{K}} \tau_D\varepsilon \leq B_{\textbf{v}+\textbf{f},\mathcal{K}} \tau_\varepsilon\,\varepsilon$, where $B_{\textbf{v}+\textbf{f},\mathcal{K}}$ is a uniform upper bound on the norm of the integrand inside the compact subset $\mathcal{K}$ whose existence is guaranteed by \thref{asmp:A}. Now, observe that $\lim_{\varepsilon\rightarrow 0} \tau_\varepsilon\,\varepsilon = 0$, and so $\forall D\in(0,\infty)$, $\exists \varepsilon_4\in(0,\varepsilon_3)$ such that $B_{\textbf{v}+\textbf{f},\mathcal{K}} \tau_\varepsilon\,\varepsilon \leq D/2,\,\forall \varepsilon\in(0,\varepsilon_4)$. Hence, we have that $\forall \varepsilon\in(0,\varepsilon_4)$, $\lVert x(\tau_D)-\tilde{\textbf{x}}(\tau_D)\lVert \leq D/2$ which contradicts the definition of $\tau_D$. Accordingly, we have that for all bounded subsets $\mathcal{B}_{\textbf{x}}\times \mathcal{B}_{\textbf{z}}\subset\mathbb{R}^n\times \mathbb{R}^m$, $\forall D\in(0,\infty)$, $\exists \varepsilon_4\in(0,\varepsilon_3)$, such that $\forall \varepsilon\in(0,\varepsilon_4)$, $\forall (\textbf{x}_0,\textbf{z}_0)\in \mathcal{B}_{\textbf{x}}\times \mathcal{B}_{\textbf{z}}$, we have that $\tau_\varepsilon < \tau_D$.
    
    Next, we show that $\exists\varepsilon_5\in(0,\varepsilon_4)$ such that $t_f/\varepsilon^2< \tau_D$ $\forall \varepsilon\in(0,\varepsilon_5)$. To obtain a contradiction, suppose that there exists a bounded subset $\mathcal{B}_{\textbf{x}}\times \mathcal{B}_{\textbf{z}}\subset\mathbb{R}^n\times \mathbb{R}^m$, a $t_f\in(0,\infty)$, and a $D\in(0,\infty)$, such that $\forall \varepsilon_5\in(0,\varepsilon_4)$, $\exists \varepsilon\in(0,\varepsilon_5)$ such that $t_f/\varepsilon^2 \geq \tau_D$. Once again, we estimate the difference $\lVert \textbf{x}(\tau)-\tilde{\textbf{x}}(\tau)\lVert$ on the interval $[0,\tau_D]$. First, for $\tau\leq\tau_\varepsilon$, observe that $\lVert \textbf{x}(\tau)-\tilde{\textbf{x}}(\tau)\lVert= \lVert \int_{0}^{\tau}\big(\sum_{i=1}^2\varepsilon^i (\textbf{v}_i(\textbf{x}(s),\textbf{z}(s),s)-\tilde{\textbf{f}}_i(\tilde{\textbf{x}}(s),s)) + O(\varepsilon^3)\big) ds\lVert\leq B_{\textbf{v}+\textbf{f},\mathcal{K}}\tau_{\varepsilon}\varepsilon$ for some constant $B_{\textbf{v}+\textbf{f},\mathcal{K}}$. Since $\lim_{\varepsilon\rightarrow 0}\tau_\varepsilon\varepsilon = 0$, we conclude that when $\tau<\tau\varepsilon$, the difference $\lVert \textbf{x}(\tau)-\tilde{\textbf{x}}(\tau)\lVert$ can be made arbitrarily small by choosing $\varepsilon$ small enough. Second, for $\tau>t_\varepsilon$, we have that:
    $\lVert \textbf{x}(\tau)-\tilde{\textbf{x}}(\tau)\lVert= \lVert \int_{0}^{\tau}\big(\sum_{i=1}^2\varepsilon^i (\textbf{v}_i(\textbf{x}(s),\textbf{z}(s),s)-\tilde{\textbf{f}}_i(\tilde{\textbf{x}}(s),s)) + O(\varepsilon^3)\big) ds\lVert \leq \lVert \int_{0}^{\tau_\varepsilon}\big(\sum_{i=1}^2\varepsilon^i (\textbf{v}_i(\textbf{x}(s),\textbf{z}(s),s)-\tilde{\textbf{f}}_i(\tilde{\textbf{x}}(s),s))\big) ds\lVert+\lVert \int_{\tau_\varepsilon}^{\tau}\big(\sum_{i=1}^2\varepsilon^i (\textbf{v}_i(\textbf{x}(s),\textbf{z}(s),s)-\tilde{\textbf{f}}_i(\tilde{\textbf{x}}(s),s))\big) ds\lVert$ $+ B_{\textbf{v}+\textbf{f},\mathcal{K}}\tau_D\varepsilon^3$, which leads to the estimate: 
    \begin{gather}\label{eq:I_estimate}
    \lVert \textbf{x}(\tau)-\tilde{\textbf{x}}(\tau)\lVert \leq B_{\textbf{v}+\textbf{f},\mathcal{K}}(\tau_\varepsilon+\tau_D\varepsilon^2)\varepsilon+ \lVert \textbf{I}_1\lVert\\
    \textbf{I}_1 = \int_{\tau_\varepsilon}^{\tau} \sum_{i=1}^2\varepsilon^i (\textbf{v}_i(\textbf{x}(s),\textbf{z}(s),s))-\tilde{\textbf{f}}_i(\tilde{\textbf{x}}(s),s)) ds
    \end{gather}
    on the interval $[0,\tau_D]$. We proceed to estimate $\lVert \textbf{I}_1 \lVert $ as follows:
    \begin{align}\label{eq:I1_estimate}
        \lVert \textbf{I}_1\lVert &\leq \varepsilon (\lVert \textbf{I}_2\lVert + \lVert \textbf{I}_3\lVert ) + \varepsilon^2 (\lVert \textbf{I}_4\lVert +\lVert \textbf{I}_5\lVert )
    \end{align}
    where $\textbf{I}_i$ for $i\in\{2,3,4,5\}$ are given by:
    \begin{align}
        \textbf{I}_2&= \int_{\tau_\varepsilon}^{\tau}(\textbf{v}_1(\textbf{x}(s),\textbf{z}(s),s))-\textbf{v}_1(\textbf{x}(s),0,s)) ds\\
        \textbf{I}_3&= \int_{\tau_\varepsilon}^{\tau} (\textbf{v}_1(\textbf{x}(s),0,s)-\tilde{\textbf{f}}_1(\tilde{\textbf{x}}(s),s)) ds\\
        \textbf{I}_4&= \int_{\tau_\varepsilon}^{\tau}(\textbf{v}_2(\textbf{x}(s),\textbf{z}(s),s))-\textbf{v}_2(\textbf{x}(s),0,s)) ds\\
        \textbf{I}_5&= \int_{\tau_\varepsilon}^{\tau} (\textbf{v}_2(\textbf{x}(s),0,s)-\tilde{\textbf{f}}_2(\tilde{\textbf{x}}(s),s)) ds
    \end{align}
    Observe that $\textbf{v}_i(\textbf{x},0,s)=\tilde{\textbf{f}}_i(\textbf{x},s)$, and so we have that:
    \begin{align}
        \textbf{I}_3&= \int_{\tau_\varepsilon}^{\tau} (\tilde{\textbf{f}}_1(\textbf{x}(s),s)-\tilde{\textbf{f}}_1(\tilde{\textbf{x}}(s),s)) ds \\
        \textbf{I}_5&= \int_{\tau_\varepsilon}^{\tau} (\tilde{\textbf{f}}_2(\textbf{x}(s),s)-\tilde{\textbf{f}}_2(\tilde{\textbf{x}}(s),s)) ds
    \end{align}
    We estimate each of the integrals above, starting by $\textbf{I}_2$, $\textbf{I}_4$ and $\textbf{I}_5$, which can be estimated as:
    \begin{align}\label{eq:I2_estimate}
        \lVert \textbf{I}_2\lVert &\leq \int_{\tau_\varepsilon}^{\tau} L_{\textbf{v}_1,\mathcal{K}} \lVert \textbf{z}(s) \lVert d\tau \\\label{eq:I4_estimate}
        \lVert \textbf{I}_4\lVert &\leq \int_{\tau_\varepsilon}^{\tau} L_{\textbf{v}_2,\mathcal{K}} \lVert \textbf{z}(s) \lVert ds \\\label{eq:I5_estimate}
        \lVert \textbf{I}_5\lVert &\leq \int_{\tau_\varepsilon}^{\tau} L_{\textbf{f}_2,\mathcal{K}} \lVert \textbf{x}(s)-\tilde{\textbf{x}}(s) \lVert ds
    \end{align}
    where $L_{\textbf{v}_1,\mathcal{K}},L_{\textbf{v}_2,\mathcal{K}},L_{\textbf{f}_2,\mathcal{K}}>0$ are Lipschitz constants. Next, we estimate $\lVert \textbf{I}_3\lVert $. We proceed by dividing the interval $\mathcal{I}=[\tau_\varepsilon,\tau]$ into sub-intervals of length $T$ and a left over piece:
    $$\mathcal{I}=\left(\bigcup_{i=1}^{k(\varepsilon)} [T_
    {i-1},T_i]\right) \bigcup \,[k(\varepsilon)T,\tau], $$ where $T_i=\tau_\varepsilon+i\,T$, and $k(\varepsilon)$ is the unique integer such that $k(\varepsilon) T \leq \tau < k(\varepsilon) T + T$. Then, we split $\textbf{I}_3$ into a sum of sub-integrals:
\begin{align}
      \textbf{I}_3
      &=\sum_{i=1}^{k(\varepsilon)} \textbf{I}_{3,i}+ \int_{k(\varepsilon)T}^{\tau}\left(\tilde{\textbf{f}}_1(\textbf{x}(s),s)-\tilde{\textbf{f}}_1(\tilde{\textbf{x}}(s),s) \right) ds \\
      \textbf{I}_{3,i}&= \int_{T_{i-1}}^{T_i}\left(\tilde{\textbf{f}}_1(\textbf{x}(s),s)-\tilde{\textbf{f}}_1(\tilde{\textbf{x}}(s),s) \right)ds
\end{align}
The part of the integral on the leftover piece can be bounded independently from $\varepsilon$ as follows:
\begin{align}
    \left\lVert\int_{k(\varepsilon)T}^{\tau}\left(\tilde{\textbf{f}}_1(\textbf{x}(s),s)-\tilde{\textbf{f}}_1(\tilde{\textbf{x}}(s),s) \right) ds\right\lVert \leq 2 B_{\textbf{f}_1,\mathcal{K}} T
\end{align}
Next, we employ Hadamard's lemma to obtain:
\begin{equation}
\begin{aligned} 
      \textbf{I}_{3,i}&=\int_{T_{i-1}}^{T_i}\textbf{F}_1(\textbf{x}(s),\tilde{\textbf{x}}(s),s) (\textbf{x}(s)-\tilde{\textbf{x}}(s)) ds
\end{aligned}
\end{equation}
where the matrix valued map $\textbf{F}_1$ is given by:
\begin{align}
&\textbf{F}_1(\textbf{x},\tilde{\textbf{x}},s)  =\int_{0}^{1}\partial_{\textbf{w}}\tilde{\textbf{f}}_1(\textbf{w},s)|_{\textbf{w}=\tilde{\textbf{x}}+\lambda(\textbf{x}-\tilde{\textbf{x}})} d\lambda
\end{align}
Through adding and subtracting a term, we may write:
\begin{equation}
\begin{aligned} 
     \textbf{I}_{3,i}=&\int_{T_{i-1}}^{T_i} \textbf{F}_1(\textbf{x}(T_{i-1}),\tilde{\textbf{x}}(T_{i-1}),s)\,(\textbf{x}(s)-\tilde{\textbf{x}}(s)) ds \\
     +&\int_{T_{i-1}}^{T_i}\Delta_i\big[\textbf{F}_1\big](s)\,(\textbf{x}(s)-\tilde{\textbf{x}}(s)) ds
\end{aligned}
\end{equation}
where the term $\Delta_i[\textbf{F}_1]$ is given by:
\begin{equation*}
\begin{aligned} 
    \Delta_i\big[\textbf{F}_1\big](s) &= \textbf{F}_1(\textbf{x}(s),\tilde{\textbf{x}}(s),s)-\textbf{F}_1(\textbf{x}(T_{i-1}),\tilde{\textbf{x}}(T_{i-1}),s)
\end{aligned}
\end{equation*}
Next, since the matrix-valued map $\textbf{F}_1$ is periodic with zero average over its third argument when the other arguments are fixed, we have that
\begin{align}
    \int_{T_{i-1}}^{T_i} \textbf{F}_1(\textbf{x}(T_{i-1}),\tilde{\textbf{x}}(T_{i-1}),s)\,\textbf{w}\,ds = 0 
\end{align}
for any fixed $\textbf{w}$. Thus, we may write:
\begin{equation}
\begin{aligned} 
    &\textbf{I}_{3,i}=\int_{T_{i-1}}^{T_i}\Delta_i\big[\textbf{F}_1\big]\,(\textbf{x}(s)-\tilde{\textbf{x}}(s)) ds \\
     &+\int_{T_{i-1}}^{T_i} \textbf{F}_1(\textbf{x}(T_{i-1}),\tilde{\textbf{x}}(T_{i-1}),s)\Delta_i[\textbf{x}-\tilde{\textbf{x}}] ds
\end{aligned}
\end{equation}
where $\Delta_i[\textbf{x}-\tilde{\textbf{x}}] = (\textbf{x}(s)-\textbf{x}(T_{i-1}))-(\tilde{\textbf{x}}(s)-\tilde{\textbf{x}}(T_{i-1}))$. The fundamental theorem of calculus yields: 
\begin{align}
    &(\textbf{x}(s)-\textbf{x}(T_{i-1}))-(\tilde{\textbf{x}}(s)-\tilde{\textbf{x}}(T_{i-1}))\\
    =&\varepsilon \int_{T_{i-1}}^{s}(\textbf{v}_1(\textbf{x}(\nu),\textbf{z}(\nu),\nu)-\tilde{\textbf{f}}_1(\tilde{\textbf{x}}(\nu),\nu))d\nu + O(\varepsilon^2) \\
    =&\varepsilon \int_{T_{i-1}}^{s}(\textbf{v}_1(\textbf{x}(\nu),\textbf{z}(\nu),\nu)-\textbf{v}_1(\textbf{x}(\nu),0,\nu))d\nu \\
    +&\varepsilon \int_{T_{i-1}}^{s}(\tilde{\textbf{f}}_1(\textbf{x}(\nu),\nu)-\tilde{\textbf{f}}_1(\tilde{\textbf{x}}(\nu),\nu))d\nu+ O(\varepsilon^2)
\end{align}
Through integration by parts, we obtain:
\begin{equation}
\begin{aligned} 
    &\int_{T_{i-1}}^{T_i}\textbf{F}_1(\textbf{x}(T_{i-1}),\tilde{\textbf{x}}(T_{i-1}),s)\Delta_i[\textbf{x}-\tilde{\textbf{x}}] ds \\
    =&\textbf{I}_{\textbf{F},i}(s)\,\Delta_i[\textbf{x}-\tilde{\textbf{x}}]\bigg|_{s=T_{i-1}}^{s=T_i}-\varepsilon\int_{T_{i-1}}^{T_i}\textbf{I}_{\textbf{F},i}(s)\,   \Delta[\tilde{\textbf{f}}_1] ds \\
    -&\varepsilon\int_{T_{i-1}}^{T_i}\textbf{I}_{\textbf{F},i}(s)\, \Delta[\textbf{v}_1]ds + O(\varepsilon^2)
\end{aligned}
\end{equation}
where we have that:
\begin{align}
    \Delta[\textbf{v}_1] &= \textbf{v}_1(\textbf{x}(s),\textbf{z}(s),s)-\textbf{v}_1(\textbf{x}(s),0,s) \\
    \Delta[\tilde{\textbf{f}}_1]&= \tilde{\textbf{f}}_1(\textbf{x}(s),s)-\tilde{\textbf{f}}_1(\tilde{\textbf{x}}(s),s) \\
    \textbf{I}_{\textbf{F},i}(s)&=\int_{T_{i-1}}^s\textbf{F}_1(\textbf{x}(T_{i-1}),\tilde{\textbf{x}}(T_{i-1}),\nu)d\nu\,
\end{align}
The boundary term coming out of the integration by parts vanishes because the right factor vanishes at $s=T_{i-1}$ and the left factor vanishes at $s=T_{i}$, leaving only the integral terms. Using Lipschitz continuity and boundedness on compact subsets, it is not hard to see that:
\begin{align}
    &\left\lVert \int_{T_{i-1}}^{T_i}\textbf{I}_{\textbf{F},i}(s)\,\Delta[\textbf{v}_1]ds \right\lVert\leq \int_{T_{i-1}}^{T_i} M_{\textbf{I}_\textbf{F},\textbf{v}_1,\mathcal{K}}\,\lVert \textbf{z}(s) \lVert ds \\
    &\left\lVert \int_{T_{i-1}}^{T_i}\textbf{I}_{\textbf{F},i}(s)\,\Delta[\tilde{\textbf{f}}_1] ds\right\lVert \leq \int_{T_{i-1}}^{T_i} M_{\textbf{I}_\textbf{F},\textbf{f}_1,\mathcal{K}}\,\lVert \Delta[\textbf{x}] \lVert ds\\
    &\left\lVert\int_{T_{i-1}}^{T_i}\Delta_i\big[\textbf{F}_1\big]\,\Delta[\textbf{x}] ds\right\lVert \leq \varepsilon\int_{T_{i-1}}^{T_i} L_{\textbf{F}_1,\mathcal{K}}\,\,\lVert \Delta[\textbf{x}] \lVert ds
\end{align}
where $\Delta[\textbf{x}]=\textbf{x}(s)-\tilde{\textbf{x}}(s)$. By utilizing the above estimates, the integral on the sub-intervals can be shown to satisfy the bound:
\begin{align}
    \left\lVert\textbf{I}_{3,i}\right\lVert \leq L_{\mathcal{K}} \,\varepsilon \,\int_{T_{i-1}}^{T_i} \left(\lVert \Delta[\textbf{x}] \lVert + \lVert \textbf{z}(s) \lVert\right)ds
\end{align}
for some Lipschitz constant $L_{\mathcal{K}}$ and consequently the integral term $\textbf{I}_3$ satisfies the bound:
\begin{align}\label{eq:I3_estimate}
    \lVert \textbf{I}_3\lVert &\leq L_{\mathcal{K}} \,\varepsilon \,\int_{\tau_\varepsilon}^{\tau} \left(\lVert \Delta[\textbf{x}] \lVert + \lVert \textbf{z}(s) \lVert\right)ds+ 2B_{\textbf{f}_1,\mathcal{K}}T
\end{align}
Combining (\ref{eq:I_estimate}), (\ref{eq:I1_estimate}), (\ref{eq:I2_estimate}), (\ref{eq:I5_estimate}), (\ref{eq:I4_estimate}), and  (\ref{eq:I3_estimate}), in addition to the fact that $\tau_\varepsilon<\tau_D,\,\forall \varepsilon\in(0,\varepsilon_4)$, we can show that the following estimate holds:
\begin{equation}
    \begin{aligned}
    &\lVert \textbf{x}(\tau)-\tilde{\textbf{x}}(\tau)\lVert \leq (M_{\mathcal{K},1}+M_{\mathcal{K},2} \tau_\varepsilon  + M_{\mathcal{K},3}\tau_D\varepsilon^2)\varepsilon \\
    &+ M_{\mathcal{K},4}\varepsilon \int_{\tau_\varepsilon}^{\tau}\lVert \textbf{z}(s)\lVert ds +M_{\mathcal{K},5}\varepsilon^2\int_{\tau_\varepsilon}^{\tau}\lVert \textbf{x}(s)-\tilde{\textbf{x}}(s)\lVert ds
    \end{aligned}
\end{equation}
Using the fact that $\lVert \textbf{z}(\tau)\lVert<2 \alpha\,\varepsilon^{\frac32},\,\forall \tau>\tau_\varepsilon$ by definition, we obtain that:
\begin{align}
    \int_{\tau_\varepsilon}^{\tau}\lVert \textbf{z}(s)\lVert ds \leq \alpha\, \tau\,\varepsilon^{\frac32} \leq \alpha\, \tau_D\,\varepsilon^{\frac32}
\end{align}
Now, remember that in order to obtain a contradiction we assumed that $\tau_D\leq t_f/\varepsilon^2$, and so we will have:
\begin{align}
\lVert \textbf{x}&(\tau)-\tilde{\textbf{x}}(\tau)\lVert \leq \delta(\varepsilon) +\int_{0}^{\tau} M_{\mathcal{K},5}\varepsilon^2 \lVert \textbf{x}(s)-\tilde{\textbf{x}}(s)\lVert ds
\end{align}
where the function $\delta(\varepsilon)$ is given by:
\begin{align}
\delta(\varepsilon)&= M_{\mathcal{K},1}\varepsilon + M_{\mathcal{K},2} \tau_\varepsilon\varepsilon + M_{\mathcal{K},3}t_f\varepsilon + M_{\mathcal{K},4}t_f \varepsilon^{\frac12}
\end{align}
An application of Gr\"onwall's inequality yields:
\begin{align}
    \lVert \textbf{x}&(\tau)-\tilde{\textbf{x}}(\tau)\lVert \leq \delta(\varepsilon) \text{e}^{M_{\mathcal{K},5}\varepsilon^2\tau} \leq \delta(\varepsilon) \text{e}^{M_{\mathcal{K},5}\varepsilon^2\tau_D}
\end{align}
on the interval $\tau\in[0,\tau_D]$. Once again, recall that we assumed that $\tau_D\leq t_f/\varepsilon^2$, and so we have: 
\begin{align}
    \lVert \textbf{x}(\tau)-\tilde{\textbf{x}}(\tau)\lVert \leq \delta(\varepsilon) \text{e}^{M_{\mathcal{K},5}t_f}
\end{align}
Now, observe that $\lim_{\varepsilon\rightarrow 0} \delta(\varepsilon) = 0$, and so we are guaranteed the existence of an $\varepsilon_5\in(0,\varepsilon_4)$ such that $\forall\varepsilon\in(0,\varepsilon_5)$ we will have:
\begin{align}
     \lVert \textbf{x}&(\tau_D)-\tilde{\textbf{x}}(\tau_D)\lVert \leq D/2
\end{align}
which contradicts the definition of $\tau_D$. Hence, the assumption that $\tau_D\leq t_f/\varepsilon^2$ is wrong, and we have that for all bounded subsets $\mathcal{B}_{\textbf{x}}\times \mathcal{B}_{\textbf{z}}\subset\mathbb{R}^n\times \mathbb{R}^m$, $\forall D\in(0,\infty)$, and $\forall t_f\in(0,\infty)$, $\exists \varepsilon_5\in(0,\varepsilon_4)$, such that $\forall \varepsilon\in(0,\varepsilon_5)$, $\forall (\textbf{x}_0,\textbf{z}_0)\in \mathcal{B}_{\textbf{x}}\times \mathcal{B}_{\textbf{z}}$, we have that $t_f/\varepsilon^2 < \tau_D$, which shows that the lemma also holds in case C1).
\end{proof}

\thref{prop:A} follows from this lemma, after reversing the time scaling $\tau=\omega (t-t_0)$ and the near identity part of the coordinate shift (\ref{eq:nearid_shift}), coupled with \cite[Theorem 2.1]{abdelgalil2022recursive} and the fact that for $\textbf{x}\in\mathcal{K}$ where $\mathcal{K}$ is any compact subset, the maps $\bm{\varphi}_i$ for $i\in\{1,2\}$ are uniformly bounded in time due to continuity and periodicity.

\end{proof}
\bibliography{refs.bib}
\bibliographystyle{IEEEtran}
\end{document}